\documentclass[12pt]{amsart}
\usepackage{amsmath,amssymb}
\usepackage{color}

\theoremstyle{plain}
\newtheorem{thm}{Theorem}[section]
\newtheorem{lem}[thm]{Lemma}
\newtheorem{cor}[thm]{Corollary}
\newtheorem{prop}[thm]{Proposition}
\newtheorem{question}[thm]{Question}
\newtheorem{problem}[thm]{Problem}

\theoremstyle{definition}
\newtheorem{remark}[thm]{Remark}
\newtheorem{example}[thm]{Example}

\begin{document}

\title{\bf Simple polynomial equations over $(2 \times 2)$-matrices}

\author{Vitalij A.~Chatyrko and Alexandre Karassev}

%\date{August 14, 2009}

\begin{abstract} 
We consider the polynomial equation
$$X^n + a_{n-1}\cdot X^{n-1} + \dots + a_1 \cdot X + a_0 \cdot I = O,$$
over $(2 \times 2)$-matrices $X$ with the real entries, where $I$ is the identity matrix, $O$ is the null matrix, $a_i \in \mathbb R$
for each $i$ and $n \geq 2$. We discuss its solution set $S$ supplied with the natural Euclidean topology.  We completely describe $S$. We also show that $\dim S =2.$
\end{abstract}

\keywords{polynomial equation over matrices; matrix algebra; covering dimension.}

\subjclass[2020]{Primary: 15A24, 54F45, Secondary: 15B30}

\maketitle

\section{Introduction}

As usual, by $M_2(\mathbb R)$  we denote the algebra of real 
$(2\times 2)$-matrices, and by $GL_2 (\mathbb R)$ the group of invertible $(2\times 2)$-matrices. We equip $M_2 (\mathbb R)$ with any of equivalent norms that make  $M_2(\mathbb R)$ a topological algebra, homeomorphic to $\mathbb R^4.$

Consider a polynomial with real coefficients  $f(x) = x^n + a_{n-1}x^{n-1} + \dots + a_1 x + a_0$ with  $n \geq 1.$

It is well known that a polynomial equation $$f(x) = 0 \eqno{(1)}$$  over the complex numbers $\mathbb C$  has exactly $n$
roots (counting multiplicities). 

Considered over the reals, the number of distinct roots of the equation $(1)$ is either between $0$ and $n$ if $n$ is even, or between $1$ and $n$, if $n$ is odd.
%In particular, if $n=2$ the equation $(1)$ over the real numbers  has the number of distinct roots  $0$, $1$ or $2$.

Next, consider a polynomial  $$F(X) = X^n + A_{n-1}X^{n-1} + \dots + A_1 X + A_0$$ of degree  $n \geq 1,$
where $X$ and $A_i$, $i=0,1,\dots, n-1,$ are from  $M_2(\mathbb R)$, and a polynomial equation
$$F(X) = O \eqno{(2)},$$
where $O$ is the null $(2\times 2)$-matrix. 

Any matrix $B \in M_2(\mathbb R)$ such that $F(B) = O$ is called {\it a solution} of $(2)$. {\it The solution set} of $(2)$ consists of all solutions of $(2)$.

It is easy to check that the matrices 
$$ \left(\begin{array}{rr}  t & \sqrt{1-t^2} \\ \sqrt{1-t^2} & -t  \end{array}\right), 
t \in [-1,1],$$ are solutions  of the equation $X^2 - I = O$, where $I$ is the identity $(2 \times 2)$-matrix.
Hence, for $n=2$ the number of distinct solutions of the equation $(2)$,  unlike to the equation $(1)$
over the real numbers, can be infinite (in this case,  it has the cardinality of continuum).

It is also readily seen that the equation $X^2 = \left(\begin{array}{rr}  0 & 1 \\ 0 & 0  \end{array}\right)$ has
no solution, the equation $X^2 = \left(\begin{array}{rr}  1 & 1 \\ 0 & 0  \end{array}\right)$ has
exactly $2$  solutions, the equation $ X^2 = \left(\begin{array}{rr}  1 & 0 \\ 0 & 2  \end{array}\right)$ has
exactly $4$ solutions.

Moreover, in \cite{W} Wilson demonstrated a method showing that the polynomial equation $(2)$ of degree $2$ 
can also have exactly $1, 3, 5$ or $6$  solutions.  The method uses eigenvalues and eigenvectors and it is due to Fuchs and Schwarz \cite{FS}. No other scenarios with finitely many solutions, except for those mentioned above, are possible. 

We simplify the polynomial  $F(X)$ by setting $A_m = a_m \cdot I$ for each $0 \leq m \leq n-1$, where each $a_m$ is a real number. Denote the polynomial  $F(X)$ in this case by $F_s(X)$ and note that $$F_s(X) = X^n + a_{n-1}\cdot X^{n-1} + \dots + a_1 \cdot X + a_0 \cdot I.$$

Consider the  equation $$F_s(X) = O\eqno{(3)}.$$
In this note, we consider the following problem.

\begin{problem}\label{problem_1} Describe the solution set $S$ of  $(3)$, where $n \geq2$.
\end{problem}

Everywhere in this note, by a space we mean a metric separable space. The covering dimension of a space $X$ will be denoted by $\dim X.$ See \cite[p. 385]{E} for the definition of $\dim$ and its main properties, such as:

\begin{itemize}
\item[(i)] {\it The countable sum theorem} \cite[Theorem 7.2.1]{E}: if a space $X$ is the union  $\cup_{i=1}^\infty X_i$, where each $X_i$ is a closed subset of $X$  with $\dim X_i \leq n, n \geq 0,$  then $\dim X \leq n,$ and 

\item[(ii)]{\it The monotone theorem} \cite[Theorem 7.3.4]{E}: for every subspace $A$ of a  space $X$ we have $\dim A \leq \dim X$.
\end{itemize}

Recall  that  $\dim \mathbb R^4 = 4$ and for any non-empty  space $X$ of cardinality less than continuum we have $\dim X = 0$. 

It is easy to see that the solution set $S$ of $(3)$ (and of $(2)$ as well) is a closed subset of $M_2(\mathbb R)$ and hence $\dim S \leq 4$. In fact, we can easily improve this estimate as follows.

\begin{prop} $\dim S \leq 3$. 

(The same is valid for the solution set of $(2)$.)
\end{prop} 

\begin{proof} Assume that some  equation
$F_s(X) = 0$, $X\in M_2(\mathbb R),$ has the solution set $S$ with  $\dim S = 4$.
Then there exists a non-empty open subset $U$ of $M_2(\mathbb R)$ such that $U \subseteq S$.
Consider a matrix $B(0) = \left(\begin{array}{rr}  a & b \\ c & d  \end{array}\right)  \in U$. Note that there exists $\epsilon >0$ such that  the matrix $B(t) = \left(\begin{array}{rr}  (a+t) & b \\ c & d  \end{array}\right)$ belongs to $U$ for each $0 \leq t \leq \epsilon$.
Further,  $F_s(B(t)) = \left(\begin{array}{rr}  t^n + p(t) & * \\ * & * \end{array}\right)$, where $p(t)$ is a polynomial  of degree $\leq n-1$, and $t^n +p(t) = 0$ for every
$0 \leq t \leq \epsilon$. But the last is impossible. We have a contradiction. Hence, $\dim S \leq 3$. 
 \end{proof}
 
In connection with this, one can ask the following. 
 
\begin{question} \label{question_2} What is $\dim S$?
\end{question}

In this note we will solve Problem  \ref{problem_1} (see Theorem \ref{character_theorem}) and answer Question \ref{question_2} (see Theorem \ref{theorem_dim}).

%%%%%%%%%%%%%%%%%%%%%%%%%%%%%%%%%%%%%%%%%%%%%%%%%%%%%%%%

\section{Auxiliary facts}

%%%%%%%%%%%%%%%%%%%%%%%%%%%%%%%%%%%%%%%%%%%%%%%%%%%%%%%%

Recall that $(2\times 2)$-matrices $A$ and $B$ with entries from a field $\mathbb F$ are {\it similar} over $\mathbb F$ if there exists a $(2\times 2)$-invertible matrix $C$ with entries from $\mathbb F$ such that $A = C^{-1}BC$, and the similarity is an equivalence relation.

For real numbers $p$ and $q$, let 

$$D(p,q) = \left(\begin{array}{rr}  p & 0 \\ 0 & q  \end{array}\right), J(p) = \left(\begin{array}{rr}  p & 1 \\ 0 & p  \end{array}\right),$$ 
and for real numbers $a$ and $b$ with $a^2 +b^2 \ne 0$, let 
$$R(a,b) = \left(\begin{array}{rr}  a & b \\ -b & a \end{array}\right), 
U(a,b) = \left(\begin{array}{rr}  a + ib & 0 \\ 0 & a-ib  \end{array}\right).$$

The following propositions are standard facts from matrix algebra (see, for example, \cite{L}).

\begin{prop} \label{auxprop_1} For any $(2\times 2)$-matrix $A\in M_2(\mathbb R)$ there exists
 an invertible $C\in M_2(\mathbb R)$ such that one of the following is true:\\
$A=C^{-1} D(p,q)C$, $A = C^{-1} J(p) C,$ or $A = C^{-1} R(a,b) C,$\\
i.e. $A$ is similar over $\mathbb R$ either $D(p,q)$, $J(p)$ or $R(a,b).$
\end{prop}

\begin{prop}\label{auxprop_3} Any $A\in M_2(\mathbb R)$ is simiilar to $R(a,b)$ over $\mathbb R$ iff $A$  is similar to $U(a,b)$ over $\mathbb C$.
\end{prop}

Let  $f_s(x) = x^n + a_{n-1}x^{n-1} + \dots + a_1 x + a_0$ be a polynomial with real coefficients $a_i, 0 \leq i \leq n-1,$ corresponding to the polynomial  $F_s(X)$.

Consider a polynomial equation $$f_s(x) = 0 .\eqno{(4)}$$

The following lemmas are evident.
\begin{lem} \label{lem_1}
\phantom{*}

$F_s(B(p,q)) =  \left(\begin{array}{rr}  f_s(p) & 0 \\ 0 & f_s(q)  \end{array}\right)$; 

 $F_s(J(p)) =  \left(\begin{array}{rr}  f_s(p) & f'_s(p) \\ 0 & f_s(p)  \end{array}\right)$; 

$F_s(U(a,b)) =  \left(\begin{array}{rr}  f_s(a+ib) & 0 \\ 0 & f_s(a-ib)  \end{array}\right)$. 

In particular, $p$ and $q$ are roots of $(4)$ iff $B(p,q)$ is a solution of $(3)$;
$p$ is a  root of $(4)$ with the multiplicity $\geq 2$ iff $J(p)$ is a solution of $(3)$;
$a\pm ib$ are roots of $(4)$ iff $F_s(U(a, b)) = O.$ 
\end{lem}

\begin{lem} \label{lem_2} If $(2\times 2)$-matrices $A$ and $B$ (with real or complex entries) are similar over $\mathbb R$ or $\mathbb C$, and $F_s(A) = O,$ then $F_s(B) = O.$

In particular, $R(a,b)$ is a solution of $(3)$ iff $a\pm ib$ are roots of $(4)$. 
\end{lem}

%%%%%%%%%%%%%%%%%%%%%%%%%%%%%%%%%%%%%%
\section{The structure of $S$}
%%%%%%%%%%%%%%%%%%%%%%%%%%%%%%%%%%%%%%%

For a given  $B\in M_2(\mathbb  R)$, define 
a map 
$$Conj_B\colon GL_2 (\mathbb R) \to M_2(\mathbb R)$$
by $Conj_B(X) = X^{-1} B X$. The following is easy to verify.

\begin{lem}\label{fibres} The map $Conj_B$ is continuous, and the image $Conj_B(GL_2 (\mathbb R))$ is $\sigma$-compact.
\end{lem}

For real numbers $p$, $q$, and real nubers $a$, $b$ with $a^2+b^2 \ne 0$,
let
$$S(D(p,q)) = Conj_{D(p,q)}(GL_2 (\mathbb R)),$$ 
$$S(J(p)) = Conj_{J(p)}(GL_2 (\mathbb R)),$$
$$S(R(a,b)) = Conj_{R(a,b)}(GL_2 (\mathbb R)).$$ 

\begin{remark} 
\begin{itemize}
\item[(i)]
For any $p \in \mathbb R$ we have $S(D(p,p)) = \{D(p,p)\}$ and 
$\dim S(D(p,p)) =0$. 
\item[(ii)] For any real numbers $p \ne q$ and a matrix 
$C(\phi) = \left(\begin{array}{rr}  \cos \phi & \sin \phi \\ -\sin \phi & \cos \phi  \end{array}\right)$
we have $$C(\phi)^{-1}B(p,q)C(\phi) = \left(\begin{array}{rr}  a(p,q, \phi) & b(p,q,\phi) \\ c(p,q,\phi) & d(p,q,\phi)  \end{array}\right),$$ where 

$a(p,q, \phi) = p \cos^2 (\phi) + q \sin^2 \phi$,
$d(p,q, \phi) = q \cos^2 (\phi) + p \sin^2 \phi$ and $b(p,q,\phi) = c(p,q,\phi) = \frac{p-q}{2} \sin 2 \phi$. 

It is easy to see that for $\phi_1, \phi_2 \in [0, \frac{\pi}{4}]$ 
$$C(\phi_1)^{-1}B(p,q)C(\phi_1) = C(\phi_2)^{-1}B(p,q)C(\phi_2)$$ iff $\phi_1 = \phi_2$.

Thus the cardinality of $S(D(p,q))$ is the
continuum.

\end{itemize}
\end{remark}

A natural question arises.
\begin{question} What is $\dim S(D(p,q))$ when $p \ne q$?
Further, what are $\dim S(J(p))$ and $\dim S(R(a,b))$?
\end{question}
We  answer this question in the next section (Corollary \ref{cor_about_S}).

Using Propositions \ref{auxprop_1}, \ref{auxprop_3}, and Lemmas \ref{lem_1}, \ref{lem_2} we get the following theorem.

\begin{thm} \label{character_theorem}The solution set $S$ of the equation $(3)$ is a disjoint union of the following sets:
$S(D(p,p))$, where $p$ is a real root of  $(4)$,
$S(J(p))$, where $p$ is a real root of $(4)$ with the multiplicity $\geq 2$,
$S(D(p,q))$, where $p < q$ are real roots of $(4)$,
$S(R(a,b))$, where $a+ib, b > 0,$ is a complex root of $(4)$.
  \end{thm}
  
  \begin{cor}\label{cor_sigma_compact}The solution set $S$ of the equation $(3)$ is $\sigma$-compact.
  \end{cor}
   Let $f_s(x) = g_s(x) h_s(x)$, where $g_s(x)$, $h_s(x)$ are polynomials  with real coefficients such that the equation $g_s(x) = 0$ has only real roots and the equation $h_s(x) = 0$ has only roots from $\mathbb C\backslash \mathbb R.$
  
   \begin{cor} The solution set of $F_s(X) =O$ is the disjoint union of the solution sets of $G_s(X) = O$ and $H_s(X) = O$.
  \end{cor}
  Let $h_s(x) = \prod_{i=1}^k h_{is}(x)$, where $h_{is}(x), i \leq k,$ are polynomials of degree $2$ with real coefficients, and the equation $h_s(x) = 0$ has only simple roots. 

 \begin{cor} The solution set of $H_s(X) =O$ is the disjoint union of the solution sets of $H_{is}(X) = O, i \leq k$.
  \end{cor}

Let $h_s(x) = (r_s(x))^k, k \geq 2$, where $r_s(x)$ is a polynomial of degree $2$ with real coefficients.

\begin{cor} The solution set of $H_s(X) =O$ is equal to the solution set of $R_s(X) = O$.
  \end{cor}

%%%%%%%%%%%%%%%%%%%%%%%%%%%%%%%%%%%%%%%%%%%%%%%%%%%%%%%%%%%%%
\section{Dimension of $S$}
%%%%%%%%%%%%%%%%%%%%%%%%%%%%%%%%%%%%%%%%%%%%%%%%%%%%%%%%%%%%

Let $X = \left(\begin{array}{rr}  a & b \\ c & d  \end{array}\right)$ be any matrix from $M_2(\mathbb R).$ Note that 
$$X^2 = \left(\begin{array}{rr}  (a^2+bc) & b(a+d) \\ c(a+d) & (bc+d^2)  \end{array}\right).$$
Consider the  equation 
$$X^2 +a_1 \cdot X + a_0 \cdot I = O.\eqno{(5)}$$

\begin{lem} \label{lem_3} Let $a_1 = 0$ and $S$ be the solution set of $(5)$.
\begin{itemize}
\item[(A)] If $a_0 < 0$ then $S$ is equal to the union
$S^A_0 \cup S^A_1 \cup S^A_2$, where 

$$S^A_0 = \{\left(\begin{array}{rr}  \sqrt{|a_0|} & 0 \\ 0 & \sqrt{|a_0|}  \end{array}\right),
\left(\begin{array}{rr}  -\sqrt{|a_0|} & 0 \\ 0 & -\sqrt{|a_0|}  \end{array}\right)\},$$
$$S^A_1 = \{\left(\begin{array}{rr}  \sqrt{|a_0|-bc} & b \\ c & -\sqrt{|a_0|-bc}  \end{array}\right): (b,c) \in A \subseteq \mathbb R^2 \},$$
$$S^A_2 = \{\left(\begin{array}{rr}  -\sqrt{|a_0|-bc} & b \\ c & \sqrt{|a_0|-bc}  \end{array}\right): (b,c) \in A \subseteq \mathbb R^2 \}$$ and $A = \{(b,c) \in \mathbb R^2: |a_0| \geq bc \}$.

\item[(B)] If $a_0 = 0$ then $S$  is equal to the union
$S^B_1 \cup S^B_2$, where 
$$S^B_1 = \{\left(\begin{array}{rr}  \sqrt{-bc} & b \\ c & -\sqrt{-bc}  \end{array}\right): (b,c) \in B \subseteq \mathbb R^2 \},$$
$$S^B_2 = \{\left(\begin{array}{rr}  -\sqrt{-bc} & b \\ c & \sqrt{-bc}  \end{array}\right): (b,c) \in B \subseteq \mathbb R^2 \}$$ and $B= \{(b,c) \in \mathbb R^2: 0 \geq bc \}$.

\item[(C)] If $a_0 > 0$ then $S$  is equal to the union
$S^C_1 \cup S^C_2$, where 
$$S^C_1 = \{\left(\begin{array}{rr}  \sqrt{-a_0-bc} & b \\ c & -\sqrt{-a_0-bc}  \end{array}\right): (b,c) \in C \subseteq \mathbb R^2 \},$$
$$S^C_2 = \{\left(\begin{array}{rr}  -\sqrt{-a_0-bc} & b \\ c & \sqrt{-a_0-bc}  \end{array}\right): (b,c) \in C \subseteq \mathbb R^2 \}$$ and $C = \{(b,c) \in \mathbb R^2: -a_0 \geq bc \}$.
\end{itemize}
Moreover, the solutions sets of (A), (B) and (C) are noncompact and topologically different. However, $\dim S = 2$ for all three cases.
\end{lem}
\begin{proof}
Suppose that $ \left(\begin{array}{rr}  a & b \\ c & d  \end{array}\right)$ is in $S.$

(A) Note that
$$
\left\{\begin{array}{l} a^2+bc= |a_0| \\ b(a+d) = 0\\c(a+d) = 0 \\  d^2+bc= |a_0|\end{array}\right.
$$
If $a+d \ne 0$ then $b=c=0$. This implies that $a^2 = d^2 = |a_0|$ and then $a= d = \sqrt{|a_0|}$ or 
$a= d = -\sqrt{|a_0|}$. Thus we get the solution subset $S^A_0$. 

\noindent 
If $a+d = 0$ then $a= -d$. Observe that $a^2 = d^2 = |a_0| -bc \geq 0$. This implies
two  solution subsets $S^A_1$ and $S^A_2$.  So $S = S^A_0 \cup S^A_1 \cup S^A_2$.

Note that the set $A$ with $\dim A = 2$ is homeomorphic to each of sets $S^A_1$ and $S^A_2$. Moreover, the sets 
$S^A_1$ and $S^A_2$ are closed in $M_2(\mathbb R)$, and so is $S^A_0$. Therefore by the countable sum theorem for $\dim$ we have  $\dim S = 2.$

(B) Note that 
$$
\left\{\begin{array}{l} a^2+bc= 0 \\ b(a+d) = 0\\c(a+d) = 0 \\  d^2+bc= 0.  \end{array}\right.
$$
If $a+d \ne 0$ then $b=c=0$. This implies that $a =d =0$. We have a contradiction and no solutions.

\noindent 
If $a+d = 0$ then $a= -d$. Observe that $a^2 = d^2 =  -bc \geq 0$. This leads to
two  solution subsets $S^B_1$ and $S^B_2$, and consequently $S = S^B_1 \cup S^B_2$.

Note that the set $B$ with $\dim B = 2$ is homeomorphic to each of the sets $S^B_1$ and $S^B_2$. Moreover, the sets 
$S^B_1$ and $S^B_2$ are closed in $M_2(\mathbb R)$. Consequently $\dim S = 2.$

(C) Note that
$$
\left\{\begin{array}{l} a^2+bc= -a_0 \\ b(a+d) = 0\\c(a+d) = 0 \\  d^2+bc= -a_0.  \end{array}\right.
$$
If $a+d \ne 0$ then $b=c=0$. This implies that $a^2 = d^2 = -a_0 < 0$. Hence, we have no solution in this case.

\noindent 
If $a+d = 0$ then $a= -d$. Observe that $a^2 = d^2 = -a_0 -bc \geq 0$. This implies
two  solution subsets $S^C_1$ and $S^C_2$.  So $S = S^C_1 \cup S^C_2$.

Note that the set $C$ with $\dim C = 2$ is homeomorphic to each of the sets $S^C_1$ and $S^C_2$. Moreover, the sets 
$S^C_1$ and $S^C_2$ are closed in $M_2(\mathbb R)$. So  $\dim S = 2.$

Note also that the set $S$ is not compact for all three cases. 
\end{proof}

\begin{remark} The solution set $S$ in (A) is equal to the disjoint union 
$$S(D(-\sqrt{|a_0|}, -\sqrt{|a_0|})) \cup S(D(\sqrt{|a_0|}, \sqrt{|a_0|})) \cup
S(D(-\sqrt{|a_0|}, \sqrt{|a_0|})).$$ In particular, $\dim S(D(-\sqrt{|a_0|}, \sqrt{|a_0|})) =2.$

The solution set $S$ in (B) is equal to the  disjoint union 
$S(D(0, 0))) \cup S(J(0)).$ In particular, $\dim S(J(0)) =2.$

The solution set $S$ in (C) is equal to 
$S(R(0, \sqrt{|a_0|})).$ In particular, $\dim S(R(0, \sqrt{|a_0|})) =2.$
\end{remark}

Let $a_1 \ne 0$ in $(5)$.
Observe that $(5)$ is equivalent to the  equation 
$$(X + \frac{a_1}{2} \cdot I)^2 + (- \frac{a_1^2}{4} + a_0) \cdot I = O .\eqno{(6)}$$ 

Put $\alpha = - \frac{a_1^2}{4} + a_0$ and $Y = X +\frac{a_1}{2} \cdot I$. Then
$(6)$ can be written as 
$$Y^2 + \alpha \cdot I = O \eqno{(7)},$$ where $Y\in M_2(\mathbb R)$.

Note that $X =  \left(\begin{array}{rr}  (y_1 -  \frac{a_1}{2}) & y_2 \\ y_3 & (y_4 - \frac{a_1}{2})  \end{array}\right) $ if 
$Y =  \left(\begin{array}{rr}  y_1 & y_2 \\ y_3 & y_4  \end{array}\right)$.

Hence using Lemma \ref{lem_3} we get the following fact.

\begin{thm} \label{thm} The solution set $S$ of $(5)$ is a closed noncompact subset of $M_2(\mathbb R)$ with $\dim S = 2$ and it is equal to 
\begin{itemize}
\item[(A)] the union
$S^A_0 \cup S^A_1 \cup S^A_2$ if $\alpha < 0$, where 

$$S^A_0 = \{\left(\begin{array}{rr}  (\sqrt{-\alpha}-  \frac{a_1}{2}) & 0 \\ 0 & (\sqrt{-\alpha}-  \frac{a_1}{2})  \end{array}\right),
\left(\begin{array}{rr}  (-\sqrt{-\alpha}-  \frac{a_1}{2}) & 0 \\ 0 & (-\sqrt{-\alpha}-  \frac{a_1}{2})  \end{array}\right)\},$$
$$S^A_1 = \{\left(\begin{array}{rr}  (\sqrt{-\alpha-bc}- \frac{a_1}{2}) & b \\ c & (-\sqrt{-\alpha-bc}-  \frac{a_1}{2})  \end{array}\right): 
-\alpha \geq bc \},$$
$$S^A_2 = \{\left(\begin{array}{rr}  (-\sqrt{-\alpha-bc}-  \frac{a_1}{2}) & b \\ c & (\sqrt{-\alpha-bc}-  \frac{a_1}{2})  \end{array}\right): 
-\alpha \geq bc\};$$
\item[(B)] the union
$S^B_1 \cup S^B_2$ if $\alpha = 0$, where 
$$S^B_1 = \{\left(\begin{array}{rr}  (\sqrt{-bc}- \frac{a_1}{2}) & b \\ c & (-\sqrt{-bc}- \frac{a_1}{2})  \end{array}\right): 0 \geq bc\},$$
$$S^B_2 = \{\left(\begin{array}{rr}  (-\sqrt{-bc}- \frac{a_1}{2}) & b \\ c & (\sqrt{-bc}- \frac{a_1}{2})  \end{array}\right): 0 \geq bc \};$$
\item[(C)]  the union
$S^C_1 \cup S^C_2$ if $\alpha > 0$, where 
$$S^C_1 = \{\left(\begin{array}{rr}  (\sqrt{-\alpha-bc}- \frac{a_1}{2}) & b \\ c & (-\sqrt{-\alpha-bc}- \frac{a_1}{2})  \end{array}\right):  -\alpha \geq bc\},$$
$$S^C_2 = \{\left(\begin{array}{rr}  (-\sqrt{-\alpha-bc}- \frac{a_1}{2}) & b \\ c & (\sqrt{-\alpha-bc}- \frac{a_1}{2})  \end{array}\right):  -\alpha \geq bc \}.$$
\end{itemize}

\end{thm}

\begin{cor} \label{cor_about_S}$\dim S(D(p,q)) = 2$, $\dim S(J(p)) = 2$, $S(R(a,b)) = 2.$ In addition to it, $S(D(p,q))$, $S(J(p))$ and $S(R(a,b))$ contain subsets homeomorphic to $\mathbb R^2$.
\end{cor}
\begin{proof}\label{dimension of S} Consider the equation $(X-pI)\cdot (X-qI) =O$ with $p < q$.
By Theorem \ref{character_theorem} the solution set $S$ of this  equation 
is the disjoint union of $S(D(p,p))$, $S(D(q,q))$ and $S(D(p,q))$.
Since $\dim S = 2$ we have $\dim S(D(p,q)) = 2$.
Similarly one can prove that $\dim S(J(p)) = 2$ and $S(R(a,b)) = 2$, using equations 
$$(X-pI)^2=O\mbox{ and } (X-aI)^2+b^2I = O,$$ respectively.
\end{proof}

Theorem \ref{character_theorem} and Corollaries \ref{cor_sigma_compact}, \ref{cor_about_S} imply
\begin{thm} \label{theorem_dim} $\dim S = 2$.
\end{thm}

%%%%%%%%%%%%%%%%%%%%%%%%%%%%%%%%%%%%%%%%%%%%%%%%%%%

\section{Some examples of solution sets}
%%%%%%%%%%%%%%%%%%%%%%%%%%%%%%%%%%%%%%%%%%%%%%%%%%

Here we will describe the solution sets of some equations. All facts in this section follow from the results of Sections~3 and~4.

\begin{prop} The solution set of $X^n =O, n \geq 2,$ is equal to the disjoint union of 
$S(D(0,0))$ and $S(J(0))$ which is  
the solution set of  $X^2 =O$.
\end{prop}

\begin{prop} The solution set of $(X^2 +I)^n=O, n \geq 1,$ is equal to 
$S(R(0,1)),$  which is  
the solution set of  $X^2  +I=O$.
\end{prop}

\begin{prop} The solution set of $X^2 -I =O$ is equal to the disjoint union of 
$S(D(-1,-1))$ , $S(D(1,1))$ and $S(D(-1,1))$.  
\end{prop}

\begin{prop} The solution set of $(X^2 -I)^n =O, n \geq 2,$ is equal to the disjoint union of 
$S(D(-1,-1))$ , $S(D(1,1))$, $S(J(-1))$, $S(J(1))$ and $S(D(-1,1))$, which is the solution set of 
$(X^2 -I)^2 =O$.  
\end{prop}

\begin{cor} The solution set $S_1$ of $X^2 -I =O$ and the solution set $S_2$ of 
$(X^2 -I)^2 =O$ are  different.  Moreover, $$S_2 = S_1 \cup S(J(-1)) \cup S(J(1))$$ and 
$$\dim (S_2 \setminus S_1) = \dim (S(J(-1)) \cup S(J(1))) =2.$$
\end{cor}

\begin{prop} The solution set  of $X^2(X^2 -I) =O$ is equal to the disjoint union of 
$S(D(0,0))$, $S(J(0))$, $S(D(-1,-1))$ , $S(D(1,1))$, $S(D(-1,1))$,  $S(D(-1,0)),$ and  $S(D(0,1))$.
\end{prop}

\begin{cor} The solution set $S$ of $X^2(X^2 -I)=O$ and the union of the solution set $S_1$ of 
$X^2 =O$ and the solution set $S_2$ of 
$X^2-I =O$ are  different.  Moreover, $$S = S_1 \cup S_2 \cup S(D(-1,0)) \cup S(D(0,1))$$ and 
$$\dim (S \setminus (S_1 \cup S_2)) = \dim (S(D(-1,0)) \cup S(D(0,1))) =2.$$
\end{cor}

%%%%%%%%%%%%%%%%%%%%%%%%%%%%%%%%%%%%%%%%%%%%%%%%%%%

\section{Examples of $F_s(X)= A$, where $A \ne O$ and $\det A = 0$}
%%%%%%%%%%%%%%%%%%%%%%%%%%%%%%%%%%%%%%%%%%%%%%%%%%

Consider $A\in M_2(\mathbb R)$ such that
$A \ne O$ and $\det A = 0$, and the equation

$$X^2 + a_0 \cdot I = A .\eqno{(8)}$$ 

\begin{prop}\label{lem_4} The solution set $S_A$ of $(8)$ has exactly $0, 2,$ or $4$  distinct solutions. In particular, if $a_0 > 0$ then $S_A = \emptyset.$
\end{prop}
\begin{proof}  Note that $(8)$ is equivalent to 
$X^2 = C,$
where $C = A - a_0 \cdot I$ and $C$ has a real eigenvalue $\lambda = -a_0$. Hence, the second eigenvalue
$\mu$ of $C$ is also  real. Consider two cases: $\lambda \ne \mu$ and $\lambda = \mu.$

If $\lambda \ne \mu$,
without loss of generality we may assume that 
$C = \left(\begin{array}{rr}  \lambda & 0 \\ 0 & \mu  \end{array}\right).$
Note that
$$
\left\{\begin{array}{l} a^2+bc= \lambda \\ b(a+d) = 0\\c(a+d) = 0 \\  d^2+bc= \mu.  \end{array}\right.
$$
If $a+d \ne 0$ then $b=c=0$. This implies that $a^2 = \lambda$, $d^2 = \mu$.  
Then if at least one of  $\lambda$, $\mu$ is negative, we have no solutions.
If $\mu > 0$ and $\lambda = 0$ or $\mu =0$ and $\lambda > 0$, we have $2$ distinct solutions, and if both  $\mu$,  $\lambda$ are positive, we have $4$ distinct solutions.

\noindent 
If $a+d = 0$ then $a= -d$ and so $a^2 = d^2$. This implies that $\lambda = \mu$, a contradiction.
We have no additional solutions. 

The case $\lambda = \mu$. Without loss of generality we may assume that
$C = \left(\begin{array}{rr}  \lambda &  1 \\ 0 & \lambda  \end{array}\right)$.
Note that
$$
\left\{\begin{array}{l} a^2+bc= \lambda \\ b(a+d) = 1\\c(a+d) = 0 \\  d^2+bc= \lambda.  \end{array}\right.
$$
Observe that $a+d \ne 0$. So  $c=0$. This implies that $a^2 = d^2= \lambda$.  
Then if $\lambda \leq 0$, we have no solutions.
If $\lambda >0$, we have $2$ distinct solutions.

\end{proof}

\begin{example} 
\begin{itemize}
\item[(i)] The equation $X^2 +I = \left(\begin{array}{rr}  0 & 1 \\ 0 & 0  \end{array}\right)$ has no solutions,
\item[(ii)] the equation $X^2 -I = \left(\begin{array}{rr}  0 & 1 \\ 0 & 0  \end{array}\right)$ has $2$ solutions,
\item[(iii)] the equation $X^2 -I = \left(\begin{array}{rr}  1 & 1 \\ 0 & 0  \end{array}\right)$ has $4$ solutions.
\end{itemize}
\end{example}

Consider  equation
$$X^2 + a_1 \cdot X + a_0 \cdot I = A.\eqno{(9)}$$ 
\begin{cor} \label{cor} If $-a_1^2 + 4 a_0 > 0$ then $(9)$ has no solution.
\end{cor}
\begin{proof} Observe that $(9)$ is equivalent to  
$$(X + \frac{a_1}{2} \cdot I)^2 + (- \frac{a_1^2}{4} + a_0) \cdot I = O.$$ 
Then apply Proposition \ref{lem_4}.
\end{proof}

\begin{prop} Let $F_s(X)= \prod_{i=1}^k H_{is}(X)$, where 

$H_{is}(X) = X^2 + a_{i1} \cdot X + a_{i0} \cdot I$ with 
$-a_{i1}^2 + 4 a_{i0} > 0$ for each $i \leq k$. 

Then the equation $F_s(X) = A$ has no solution.
\end{prop}
\begin{proof}  Let $X^*$ be a solution of the equation $F_s(X) = A$. Then
$F_s(X^*)= \prod_{i=1}^k H_{is}(X^*) = A$. Let $H_{ks}(X^*) = A_1$. Note that $A_1 \ne O$. 
If $\det A_1 = 0$ we have a contradiction with Corollary \ref{cor}.  If $\det A_1 \ne 0$, we have  
$\prod_{i=1}^{k-1} H_{is}(X^*) = A_2$ and $A_2 \ne O$, $\det A_2 = 0$. Apply induction.
\end{proof}

\begin{example} Consider the equation $$(X^2-I)^2 =\left(\begin{array}{rr}  1 & 0 \\ 0 & 0  \end{array}\right) \ \ \ \ \ \ \ \ \ \ (10).$$ Note that the equation $Y^2 = \left(\begin{array}{rr}  1 & 0 \\ 0 & 0  \end{array}\right)$
has precisely two solutions:  $\left(\begin{array}{rr}  1 & 0 \\ 0 & 0  \end{array}\right)$ and 
$\left(\begin{array}{rr}  -1 & 0 \\ 0 & 0  \end{array}\right).$  
Then consider the equations 

$X^2 - I = \left(\begin{array}{rr}  1 & 0 \\ 0 & 0  \end{array}\right)$ and $X^2 - I = \left(\begin{array}{rr}  -1 & 0 \\ 0 & 0  \end{array}\right)$.

By the proof of Proposition \ref{lem_4}, the first one has exactly $4$ solutions but the second has $2$ solutions.
So the equation $(10)$ has precisely $6$ solutions.
\end{example}

\section{Concluding remarks}

In the Introduction  we noted that there are equations $(2)$ with finite solution sets and the empty
solution set.
Hence, the solution set of $(2)$ can also have dimension $0$ and $-1$.

\begin{question} Is there an equation $(2)$ for which the solution set has dimension $1$ or $3$?
\end{question} 

In this note, we restricted our attention to $(2\times 2)$-matrices. However, one can ask the following.

\begin{question} Let $S$ be the solution set of $(3)$ considered over $(m\times m)$-matrices with real coefficients, where $m>2$. What is the structure of $S$? What is the dimension of $S$?
\end{question}

\vskip 0.5 cm
\noindent(V.A. Chatyrko)\\
Department of Mathematics, Linkoping University, 581 83 Linkoping, Sweden.\\
vitalij.tjatyrko@liu.se

\vskip0.3cm
\noindent(A. Karassev) \\
Department of Mathematics,
Nipissing University, North Bay, Canada\\
alexandk@nipissingu.ca

\end{document}